\documentclass[12pt,reqno,fleqn]{amsart}  
\usepackage{tikz}
\usepackage{amsmath,amstext,amsthm,amssymb,amsxtra}
\usepackage{mathrsfs}%for special fonts
\usepackage[top=1.5in, bottom=1.5in, left=1.25in, right=1.25in]	{geometry}
\usepackage{lmodern}

\usepackage{graphics}
\usepackage{euscript}
\usepackage{xcolor}

\usepackage{mathtools}
\mathtoolsset{showonlyrefs,showmanualtags}

\usepackage{hyperref} %,hypertexnames=false,colorlinks,[pagebackref]
\hypersetup{
	colorlinks=true,       % false: boxed links; true: colored links
	linkcolor=blue,          % color of internal links
	citecolor=magenta,        % color of links to bibliography
	filecolor=magenta,      % color of file links
	urlcolor=cyan           % color of external links
}

\usepackage[msc-links]{amsrefs}

\newtheorem{theorem}{Theorem}[section]

\newtheorem{proposition}{Proposition}[section]
\newtheorem{lemma}{Lemma}[section]
\newtheorem{corollary}{Corollary}[section]
\newtheorem{OldTheorem}{Theorem}
\theoremstyle{definition}
\newtheorem{definition}{Definition}[section]
\theoremstyle{definition}

\newtheorem{remark}{Remark}[section]

\theoremstyle{remark}

\def\BMO{{\rm BMO }}

\def\BO{{\rm BO}}

\def\OSC{{\rm OSC}}

\def\SUP{{\rm SUP}}
\def\INF{{\rm INF}}
\def\supp{{\rm supp\,}}

\def\esssup{{\rm esssup\, }}
\def\essinf{{\rm essinf\, }}
\def\dist{{\rm dist}}

\def\MM^d{\ensuremath{\mathfrak M}}
\def\MM{\ensuremath{\mathcal M}}
\def\ZM{\ensuremath{\mathfrak M}}
\def\ZB{\ensuremath{\mathscr B}}
\def\zB{\ensuremath{\mathfrak B}}
\def\ZA{\ensuremath{\mathscr A}}

\def\ZZ{\ensuremath{\mathbb Z}}

\def\ZQ{\ensuremath{\mathcal Q}}
\def\ZI{\ensuremath{\textbf 1}}

\def\ZK{\ensuremath{\mathcal K}}

\def\ZR{\ensuremath{\mathbb R}}

\def\ZT{\ensuremath{\mathbb T}}
\def\pr{\ensuremath{\mathrm {pr}}}
\def\ZL{\ensuremath{\mathcal L}}

\def\ZG{{\mathscr G\,}}

\numberwithin{equation}{section}
\def\md#1#2\emd{\ifx0#1
	\begin{equation*} #2 \end{equation*}\fi  %  single line display, no number
	\ifx1#1\begin{equation}#2\end{equation}\fi   % single line display, number
	\ifx2#1\begin{align*}#2\end{align*}\fi   % aligned display, no number
	\ifx3#1\begin{align}#2\end{align}\fi    % aligned display, number
	\ifx4#1\begin{gather*}#2\end{gather*}\fi  % multline, not align, no number
	\ifx5#1\begin{gather}#2\end{gather}\fi   % multinline, not align
	\ifx6#1\begin{multline*}#2\end{multline*}\fi  %  display too long for one line
	\ifx7#1\begin{multline}#2\end{multline}\fi  % as above, with numbers
	\ifx8#1\begin{multline*}\begin{split}#2\end{split}\end{multline*}\fi
	\ifx9#1\begin{multline}\begin{split}#2\end{split}\end{multline}\fi
}
\newcommand {\e }[1]{\eqref{#1}}
\newcommand {\lem }[1]{Lemma \ref{#1}}
\newcommand {\rem }[1]{Remark \ref{#1}}
\newcommand {\cor }[1]{Corollary \ref{#1}}
\newcommand {\pro }[1]{Proposition \ref{#1}}

\newcommand {\trm }[1]{Theorem \ref{#1}}

\newcommand {\sect }[1]{Section \ref{#1}}
\newcommand {\df }[1]{Definition \ref{#1}}

\title[] {Bounded oscillation operators on BMO spaces}
\author{Grigori A. Karagulyan}
\address{Institute of Mathematics of NAS of RA, Marshal Baghramian ave., 24/5, Yerevan, 0019, Armenia} 
\address{Faculty of Mathematics and Mechanics, Yerevan State
University, Alex Manoogian, 1, 0025, Yerevan, Armenia} 
\email{g.karagulyan@ysu.am}

\thanks{The work was supported by the Higher Education and Science Committee
	of RA, in the frames of the research project 21AG‐1A045}

\begin{document}

\begin{abstract}
	Bounded Oscillation ($\BO$) operators were recently introduced in \cite{Kar3}, where it was proved that many operators in harmonic analysis (Calder\'on-Zygmund operators, Carleson type operators, martingale transforms, Littlewood-Paley  square functions, maximal operators, etc) are $\BO$ operators. $\BO$ operators are defined on abstract measure spaces equipped with a basis of abstract balls. The abstract balls in their definition owe four basic properties of classical balls in $\ZR^n$, which are crucial in the study of singular operators on $\ZR^n$. Among various properties studied in these papers it was proved that $\BO$ operators allow pointwise sparse domination, establishing the $A_2$-conjecture for those operators. In the present paper we study boundedness properties of $\BO$ operators on $\BMO$ spaces. In particular, we prove that general $\BO$ operators boundedly map $L^\infty$ into $\BMO$, and under a logarithmic localization condition those map $\BMO$ into itself. We obtain these properties as corollaries of new local type bounds, involving oscillations of functions over the balls. We apply the results in the $\BMO$ estimations of Calder\'on-Zygmund  operators, martingale transforms, Carleson type operators, as well as  in the unconditional basis properties of general wavelet type systems in atomic Hardy spaces $H^1$.
\end{abstract}

	\maketitle  
%%%%%%%%%%%%%%%%%%%%%%%%%%%%%% SECTION  SECTION SECTION
%%%%%%%%%%%%%%%%%%%%%%%%%%%%%% SECTION  SECTION SECTION

 \tableofcontents		
\section{Introduction}\label{S4}
\subsection{Operators on $\BMO$ spaces}
The space of functions of bounded mean oscillation, $\BMO(\ZR^d)$, was introduced by John and Nirenberg \cite{JoNi}. This space is the collection of locally integrable functions  $f$ on $\ZR^d$ such that
\begin{equation}\label{x0}
	\|f\|_{\BMO}=\sup_{B}\frac{1}{|B|}\int_B|f-f_B|<\infty,
\end{equation}
where $f_B$ denotes the mean of $f$ over a ball $B\subset \ZR^d$, and the supremum is taken over all the balls $B$. This is not properly a norm, since any function which is constant almost everywhere has zero oscillation. However, modulo constant, $\BMO$ becomes a Banach space with respect to the norm \e{x0}. The space $\BMO$ plays an important role in harmonic analysis and in the study of partial differential equations. One can check that
$L^\infty(\ZR^d)\subset \BMO(\ZR^d)$, but there exist unbounded $\BMO$ functions. However, the fundamental inequality of John-Nirenberg \cite{JoNi} states that $\BMO$ functions have exponential decay over the balls. Namely, for any ball $B\subset \ZR^d$ we have
\begin{equation}\label{JN}
	|\{x\in B:\, |f(x)-f_B|>\lambda\}|\le c_1|B|\exp(-c_2\lambda/\|f\|_\BMO),\quad \lambda>0,
\end{equation}
where $c_1$ and $c_2$ are positive absolute constants. Some remarkable results about $\BMO$ are the duality relationship between $\BMO$ and the Hardy space $H^1$ due to C.~Fefferman \cite{Fef} and a deep connection between the Carleson measures and the $\BMO$ functions established by C.~Fefferman and E.~Stein \cite{FeSt}.
There are different characterizations and generalizations of $\BMO$ spaces. 

It is of interest the problem of boundedness of certain operators in harmonic analysis on $\BMO$ spaces.  The fact that ordinary Calderón–Zyg\-mund operators map boundedly $L^\infty$ into $\BMO$ was independently obtained by Peetre \cite{Pee}, Spanne \cite{Spa} and Stein \cite {Ste}. Peetre \cite{Pee} also observed that translation-invariant Calderón–Zygmund operators actually map $\BMO$ to itself. Bennett-DeVore-Sharpley in \cite{BDS} proved that the Hardy-Littlewood uncentered maximal function is bounded on $\BMO(\ZR^d)$.

\subsection{Bounded oscillation operators}
Bounded oscillation ($\BO$) operators were recently introduced in \cite{Kar3} and it was proved that many operators in harmonic analysis (Calder\'on-Zygmund operators, maximally modulated singular operators, Carleson type operators, martingale transforms, Littlewood-Paley  square functions, maximal operators, etc) are $\BO$ operators. Various properties of $\BO$ operators were studied in \cite{Kar1,Kar3} (see also \cite{Ming} for a multilinear generalization of $\BO$ operators and those properties). In the present paper we consider the problem of boundedness of $\BO$ operators on $\BMO$ spaces.
General $\BO$ operators are defined on abstract measure spaces equipped with a ball-basis. The definition of the ball-basis is a selection of basic properties of classical balls (or cubes) in $\ZR^n$ that are crucial in the study of singular operators on $\ZR^n$. 

\begin{definition} Let $(X,\ZM, \mu)$ be a measure space with a $\sigma$-algebra $\ZM$ and a measure $\mu$. A family of measurable sets $\ZB$ is said to be a ball-basis if it satisfies the following conditions: 
	\begin{enumerate}
		\item[B1)] $0<\mu(B)<\infty$ for any ball $B\in\ZB$.
		\item[B2)] For any points $x,y\in X$ there exists a ball $B\ni x,y$.
		\item[B3)] If $E\in \ZM$, then for any $\varepsilon>0$ there exists a (finite or infinite) sequence of balls $B_k$, $k=1,2,\ldots$, such that $\mu(E\bigtriangleup \cup_k B_k)<\varepsilon$.
		\item[B4)] For any $B\in\ZB$ there is a ball $ B^*\in\ZB $ (called {\rm hull-ball} of $B$), satisfying the conditions
		\begin{align}
			&\bigcup_{A\in\ZB:\, \mu(A)\le 2\mu(B),\, A\cap B\neq\varnothing}A\subset  B^*,\label{h12}\\
			&\qquad\qquad\mu\left(B^*\right)\le \ZK\mu(B),\label{h13}
		\end{align}
		where $\ZK$ is a positive constant independent of $B$. 
	\end{enumerate}
\end{definition}
Ball-bases can additionally have the following properties.
\begin{definition}\label{DD}
	We say that a ball-basis $\ZB$ is doubling if there is a constant $\eta>2$ such that for any ball $B$ with $B^*\neq X$ one can find a ball $B'\supset B$ such that
	\begin{align}
		2\mu(B)\le \mu(B')\le\eta  \cdot \mu(B).\label{h73}
	\end{align}
\end{definition}
\begin{definition}\label{D3}
	A ball basis $\ZB$ is said to be regular if there is a constant $0<\theta<1$ such that for any two balls $B$ and $A$, satisfying $\mu(B)\le \mu(A)$, $B\cap A\neq \varnothing$ we have $\mu(B^*\cap A)\ge \theta \mu(B^*)$.
\end{definition}

Here are some basic examples of ball-bases:
\begin{enumerate}
	\item The family of Euclidean balls (or cubes) in $\ZR^d$ is a doubling ball-basis.  
	\item Dyadic cubes in $\ZR^d$ form a doubling ball-basis.
	\item Let $\ZT= \ZR/2\pi$ be the unit circle. The family of arcs in $\ZT$ is a doubling ball-basis.
	\item The families of metric balls in measure spaces of homogeneous type form a ball-basis with the doubling condition (see \cite{Kar3}, sec. 7,  for the proof)
	\item Let $(X,\ZM, \mu)$ be a measure space and $\{\ZB_n:\, n\ge 0\}$ be collections of measurable sets (called filtration) such that
	\begin{itemize}
		\item $\ZB_0=\{X\}$ and each $\ZB_n$ forms a finite or countable partition of $X$,
		\item each $A\in \ZB_n$ is a union of some sets of $\ZB_{n+1}$,
		\item the collection $\ZB=\cup_{n\in\ZZ}\ZB_n$ generates the $\sigma$-algebra $\ZM$,
		\item for any points $x,y\in X$ there is a set $A\in X$ such that $x,y\in A$.
	\end{itemize}
	For any $A\in \ZB_n$ we denote by $\pr(A)$ the parent-ball of $A$, that is the unique element of $\ZB_{n-1}$, containing $A$. One can easily check that $\ZB$ satisfies  the ball-basis conditions B1)-B4), where for $A\in \ZB$ the hull-ball $A^*$ is the maximal element in $\ZB$, satisfying $A^*\supset A$ and $\mu(A^*)\le 2\mu(A)$. Observe that $A^*$ can coincide with $A$ and it occurs when $\mu(\pr(A))>2\mu(A)$.
	Besides, a martingale ball-basis is doubling if and only if $\mu(\pr(A) )\le c\mu(A)$ for any $A\in \ZB$ and for some constant $c>1$.
\end{enumerate}
We also note that all these ball-bases are regular (see \df{D3})

\subsection{Notations}
Let $(X,\ZM, \mu)$ be a measure space equipped with a ball-basis $\ZB$. Fix $1\le r<\infty $  and denote by $L^r_{loc}(X)$ the space of local $L^r$-integrable functions $f$ on $X$, namely, $\int_B|f|^r<\infty$ for any ball $B\in \ZB$. Denote by $L^0(X)$ the space of almost everywhere finite measurable functions on $X$. We set 
\begin{align}
	&f_B=\frac{1}{\mu(B)}\int_Bf,\\
	&	\langle f\rangle_B=\left(\frac{1}{\mu(B)}\int_B|f|^r\right)^{1/r}, \\
	&\langle f\rangle_B^*=\sup_{A\in \ZB:A\supset B}\langle f\rangle_{A},\label{y56}\\
	&\langle f\rangle_{\#,B}=\langle f-f_B\rangle_B=\left(\frac{1}{\mu(B)}\int_B|f-f_B|^r\right)^{1/r},
\end{align}
\begin{align}
	& \langle f\rangle_{\#, B}^*=\sup_{A\in \ZB:A\supset B}\langle f\rangle_{\#,A},\label{x5}\\
	&\SUP_E(f)=\esssup_{x\in E}|f(x)|,\\
	& \INF_E(f)=\essinf_{x\in E}|f(x)|,\label{y98}\\
	&\OSC_E(f)=\esssup_{x,x'\in E}|f(x)-f(x')|=\SUP_E(f)-\INF_E(f),\label{y97}\\
	&\OSC_{B,\alpha}(f)=\inf_{E\subset B:\, \mu(E)> \alpha\mu(B)}\OSC_E(f),\quad 0<\alpha <1,\label{y100}
\end{align}
where  $f\in L^r_{loc}(X)$,  $A, B\in \ZB$ and $E\subset X$ is a measurable set. Note that the $\alpha$-oscillation functional $\OSC_{B,\alpha}(f)$ can be equivalently defined as follows: that is $\OSC_{B,\alpha}(f)=\inf_{a<b}(b-a)$, where the infimum is taken over all the numbers $a<b$ satisfying $\mu\{x\in B:\, f(x)\in [a,b]\}>\alpha\mu(B)$. 

In the sequel positive constants depending only on our ball-basis $\ZB$ and the parameter $1\le r<\infty$ will be called admissible constants. The relation $a\lesssim b$ will stand for $a\le c\cdot b$, where $c>0$  is admissible. We write $a\sim b$ if the relations  $a\lesssim b$ and $b\lesssim a$ hold simultaneously. The notation $\log a$ will stand for $\log_2 a$. 
\subsection{Main results}
We let $(X,\mu)$ be a measurable  space equipped with a doubling ball-basis $\ZB$. We will consider a subclass of $\BO$ operators on $X$, which definition involves certain modulus of continuity. An operator $T:L^r(X)\to L^0(X)$ is called sublinear if it satisfies $|T(f+g)|\le |T(f)|+|T(g)|$.

\begin{definition}
	Let $\omega:[1,\infty)\to [0,1]$ be a non-increasing function and suppose $r\ge 1$. We say $T:L^r(X)\to L^0(X)$ is $\BO_\omega$ operator if it is sublinear, i.e. $|T(f+g)|\le |T(f)|+|T(g)|$, and for any function $f\in L^r(X)$ and a ball $B$ we have
	\begin{align}
		\OSC_B\big(T (f\cdot \ZI_{X\setminus B^*})\big) \le \ZL_{\omega}(T)\sup_{A\in \ZB:A\supset B}\bigg(\omega\left(\frac{\mu(B)}{\mu(A)}\right)\cdot \langle f\rangle_{A}\bigg),	\label{y83}
	\end{align}
where $ \ZL_{\omega}(T)$ is a constant, depending only on $\omega$ and the operator $T$. The class of operators satisfying \e{y83} will be denoted by $\BO_\omega(X)$ and we suppose that $\ZL_{\omega}(T)$ is the least constant such that \e{y83} holds for all the functions $f\in L^r(X)$ and balls $B$. 
\end{definition}
\begin{remark}
	Condition \e{y83} is an extended version of the localization property, that was considered in the original definition of $\BO$ operators in \cite{Kar3}. Moreover,  if $\omega(t)\equiv 1$, then  \e{y83} exactly gives the definition of $\BO$ operators on measure spaces with a doubling ball-basis, since in that case the right hand side of \e{y83} becomes $\ZL(T)\langle f\rangle_B^*=\ZL_{\omega}(T)\langle f\rangle_B^*$ (see \cite{Kar3}). In that case we will use notation $\BO(X)$ instead of $\BO_\omega(X)$.
\end{remark}
\begin{definition}\label{D2}
	 We say that an operator $T$ is vanishing on constants (or simply vanishing) if for any ball $B\in \ZB$ and any $\varepsilon>0$ there exists a ball $B'\supset B$ such that for any ball $B''\supset B'$ the inequality $\OSC_B\big(T (\ZI_{B''})\big)\le \varepsilon$ holds. A family of vanishing operators $T_\alpha :L^r(X)\to L^0(X)$ is said to be uniformly vanishing if the corresponding ball $B'$ can be chosen independently of $T_\alpha$.
\end{definition}
Our main result reads as follows.
\begin{theorem}\label{T5}
	Let $\ZB$ be a doubling ball-basis in a measure space  $X$ and suppose that $T_\alpha\in \BO_\omega(X)$ is a family of uniformly vanishing linear operators with $\omega(t)= \log^{-1}(1+t)$. If $T$ is one of these operators  
	\begin{equation}\label{r50}
		T^+(f)=\sup_\alpha T_\alpha (f), \quad T^-(f)=\inf_\alpha T_\alpha (f),\quad T^*(f)=\sup_\alpha |T_\alpha (f)|,
	\end{equation}
	and for a number $1\le r<\infty$ we have
		\begin{equation}\label{y84}
	\|T\|_{L^r(X)\to L^{r,\infty}(X)}<\infty,\quad \sup_\alpha\ZL_{\omega}(T_\alpha)<\infty,
	\end{equation}
	 then the bound
	\begin{equation}\label{y30}
		\OSC_{B,\beta}\,T( f)\lesssim \left((1-\beta)^{-1/r}	\|T\|_{L^r(X)\to L^{r,\infty}(X)}+\sup_\alpha \ZL_{\omega}(T_\alpha)\right) \langle f\rangle^*_{\#,B}
	\end{equation}
	holds for any $f\in L^r(X)$, any ball $B\in \ZB$ and a number $0<\beta<1$.
\end{theorem}
{\begin{remark}
	Note that the admissible constant on the right-hand side of \e{y30} is independent of the parameter $\beta$. The same applies to the theorems presented below.
\end{remark}}
\begin{remark}
	Observe that if $\mu(X)<\infty$ and $T(\ZI_X)$ is a constant function, then $T$ is trivially vanishing, since one can simply choose $B'=X$ in \df {D2} (by \lem{L0} the condition $\mu(X)<\infty$ implies that $X$ itself is a ball, i.e. $X\in\ZB$). Thus if $\mu (X)<\infty$, then in \trm{T5} the condition to be uniformly vanishing can be replaced by the condition $T_\alpha(\ZI_X)\equiv c_\alpha$.
\end{remark}
\begin{remark}\label{R1}
	Observe that, in general, operators \e{r50} can produce non-measurable functions, since the number of operators $T_\alpha$ can be non-countable.  So in some cases we will use outer measure $\mu^*$ generated from $\mu$ (see definition in \sect{S1}). In particular, the weak $L^{r,\infty}$ norm of a function $f$ is defined in \e{r54}.
This justifies the consideration of $\|T\|_{L^r(X)\to L^{r,\infty}(X)}$ in \trm{T5}. Besides, the functions, involved in the above notations are allowed to be non-mensurable.  
\end{remark}
\begin{theorem}\label{T1}
	Let $\ZB$ be a doubling ball-basis in a measure space  $X$ and suppose that $\{T_\alpha\}\subset  \BO(X)$ is a family of general $\BO$ operators ($\omega(t)\equiv 1$ in \e{y83}). If $T$ is one of the operators \e{r50}
	and it satisfies conditions \e{y84}, then the bound
	\begin{equation}\label{y44}
		\OSC_{B,\beta}\,T( f)\lesssim \left((1-\beta)^{-1/r}	\|T\|_{L^r(X)\to L^{r,\infty}(X)}+\sup_\alpha \ZL_{\omega}(T_\alpha)\right) \|f\|_\infty
	\end{equation}
	holds for any $f\in L^\infty (X)$, any ball $B\in \ZB$ and a number $0<\beta<1$.
\end{theorem}
\begin{corollary}\label{KC2}
	Under the hypothesis of \trm{T5} if $f\in \BMO(X)\cap L^r(X)$, then $T(f)\in \BMO(X)$ and 
	\begin{equation}
		\|T(f)\|_{\BMO(X)}\lesssim \|f\|_{\BMO(X)}.
	\end{equation}
\end{corollary}
\begin{corollary}\label{KC1}
	Under the hypothesis of \trm{T1} if $f\in L^\infty(X)$, then $T(f)\in \BMO(X)$ and 
	\begin{equation}\label{y45}
		\|T(f)\|_{\BMO(X)}\lesssim \|f\|_\infty.
	\end{equation}
\end{corollary}
\cor{KC2} is immediately obtained by a combination of \trm{T5} and the following John-Nirenberg type \trm{T11}. For its statement,  first, recall the definitions of $\BMO_\alpha$ and $\BMO$ spaces as families of functions $f\in L^0(X)$, satisfying
\begin{align}
	&\|f\|_{\BMO_\alpha}=\sup_{B\in \ZB}\OSC_{B,\alpha}(f)<\infty,\\
	&\|f\|_{\BMO}=\sup_{B\in \ZB}\langle f\rangle_{\#,B}=\sup_{B\in \ZB}\langle f-f_B\rangle_{B}<\infty,\\
\end{align}
respectively. One can easily check that
\begin{align*}
	\|f\|_{\BMO_\alpha}\le 2(1-\alpha)^{-1}\|f\|_\BMO,
\end{align*}
whereas the following is a more delicate property of these norms.
\begin{theorem}\label{T11}
	There is an admissible constant $\alpha\in (1/2,1)$ such that  $ \BMO_\alpha(X)= \BMO(X)$. Moreover, for any $f\in \BMO_\alpha(X)$ we have $\|f\|_{\BMO}\sim	\|f\|_{\BMO_\alpha}$.
\end{theorem}

\begin{remark}
	\trm{T11} is an extension of the analogous result on $\ZR^d$ proved in \cite{John, Str}. The proof of \trm{T11} is based on \pro{P} (see \sect{S5}), which is itself interesting. In particular, from \pro{P} it follows that:
	if $f\in \BMO(\ZR^d)$, then for any ball $B\subset \ZR^d$,
	\begin{align}
		|\{x\in B:\, |f(x)-f_B|>&c(n+1)\|f\|_\BMO\}|\\
		&\le \frac{1}{2}\cdot |\{x\in B:\, |f(x)-f_B|>cn\|f\|_\BMO\}|\label{x57}
	\end{align}
	where $c>0$ is an absolute constant. Clearly, this implies the classical John-Nirenberg inequality \e{JN},  but the converse statement is not true at all.
\end{remark}
\begin{remark}
	Note that for the proof of \trm{T5} the regularity condition (see \df {D3}) of the ball-basis is not used, but it is significant in the proofs of \trm{T11} and  \cor{KC2}. 
\end{remark}
The author is grateful to the anonymous referee for valuable comments and suggestions that helped to improve the quality of the paper.

\section{Outer measure and $L^p$-norms of non-measurable functions}\label{S1}
It is well known for classical examples of ball-bases that the union of any collections of balls is measurable and such a property is important to ensure measurability of certain maximal operators. For general ball-bases non-countable union of balls need not to be measurable. So in some cases we will need to use outer measure generated by the basic measure $\mu$. Hence, let $(X,\ZM,\mu)$ be a measure space. Define the outer measure of an arbitrary set $E\subset X$ by
\begin{equation*}
	\mu^*(E)=\inf_{F\in \ZM:\, F\supset E}\mu(F).
\end{equation*}
{Clearly, for any set $F\subset X$ there exists a measurable cover for $F$, i.e. a set $E\in \ZM$ such that 
\begin{equation}\label{x1}
	E\supset F,\quad \mu^*(F)=\mu(E). 
\end{equation}
In general, a measurable cover is not uniquely determined. However, if $E$ and $E'$ are two measurable covers for a set $F$, then those coincide a.e., i.e. $\mu(E\triangle E')=0$. For any function $f:X\to \ZR$ (possibly non-measurable) we denote
\begin{align}
	&G_f(t)=\{x\in X:\, |f(x)|>t\},\quad t\ge 0,\label{x3}\\
	& \lambda_f(t)=\mu^*\left(G_f(t)\right).
\end{align}
Clearly $\lambda_f(t)$ is a distribution function, i.e. it is right-continuous. 
\begin{proposition}\label{P1}
	For any function $f:X\to \ZR$ (possibly non-measurable) there exists a measurable function $\bar f$, called a cover function for $f$, such that
	\begin{align}
		&|f(x)|\le \bar f(x),\quad x\in X,\label{x4}\\
		&\lambda_f(t)=\lambda_{\bar f}(t),\quad t\in \ZR.\label{x6}
	\end{align}
The function $\bar f$ is uniquely defined a.e.; that is, if $g$ is another cover function for $f$, then $g(x)=\bar f(x)$ a.e..
\end{proposition}
\begin{proof}
	Let $\overline G_f(t)$ be measurable covers for the sets \e{x3}. Making suitable choice of those cover sets and using the countability of the rationals $\mathbb Q$, one can ensure $\overline G_f(r)\supset \overline G_f(r')$ for all pairs of rationals $r<r'$. Set
	\begin{equation}
		\bar f(x)=\sup\left\{r\in \mathbb Q:\, x\in \overline G_f(r)\right\}.
	\end{equation}
	One can check that
	\begin{equation*}
		\left\{x\in X:\, f(x)>t\right\}\subset \bigcup_{r>t,\, r\in \mathbb Q}\overline G_f(r)=\left\{x\in X:\, \bar f(x)>t\right\}.
	\end{equation*}
	These relations easily imply that $\bar f$ is measurable and satisfies \e{x4}, \e{x6}. It is also easy to observe that $\bar f$ is uniquely defined a.e..
\end{proof}
\begin{remark}
	Form \pro {P1} it follows that the cover function of any function $f$ is the least measurable function dominating $f$. Namely, if $g(x)$ is measurable and $f(x)\le g(x)$ then $g(x)\ge \bar f(x)$ a.e.. This in particular implies that
	\begin{equation}\label{x54}
		\overline {f+g}\le \bar f+\bar g,\quad \overline{fg}\le \bar f\cdot \bar g.
	\end{equation}
\end{remark}
For arbitrary function $f$ on $X$ we define
\begin{align}
	&\|f\|_{L^p}=\|\bar f\|_{L^p}=\left(p\int_0^\infty t^{p-1}\lambda_f(t)dt\right)^{1/p},\label{r52}\\
	&\|f\|_{L^{p,\infty}}=\|\bar f\|_{L^{p,\infty}}=\sup_{t>0}t(\lambda_f(t))^{1/p}.\label{r54}
\end{align}
Using \pro {P1}, one can observe that the standard triangle and  H\"{o}lder inequalities hold in such setting of $L^p$ norms. In particular, if $f$ and $g$ are arbitrary functions, then, using \e{x54}, we can write
\begin{equation*}
	\|f+g\|_p=	\|\overline{f+g}\|_p\le \|\bar f+\bar g\|_p\le \|\bar f\|_p+\|\bar g\|_p=\|f\|_p+\|g\|_p
\end{equation*}
that gives the triangle inequality for general functions. Similarly it can be proved the  H\"{o}lder's inequality. As we have mentioned in \rem{R1} some maximal operators can produce non-measurable functions. One can state a version of that Marcinkiewicz interpolation theorem for such operators. Denote by $L^{0,*}(X)$ the space of arbitrary functions $f:X\to \ZR$.  
\begin{definition}
	We say a subadditive operator $T:L^p(X)\to L^{0,*}(X)$ satisfies weak-$L^p$ or strong-$L^p$ estimate if
	\begin{align*}
		&\|T\|_{L^p\to L^{p,\infty}}=\sup_{t>0,\,f\in L^p(X)}\frac{t\cdot ( \lambda_{T(f)}(t))^{1/p}}{\|f\|_{L^p}}<\infty,\\
		&\|T\|_{L^p}=\sup_{f\in L^p(X)}\frac{\|T(f)\|_{L^p}}{\|f\|_{L^p}}<\infty,
	\end{align*}
	respectively. 
\end{definition}
The proof of the following generalized Marcinkiewicz interpolation theorem is exactly the same as in the classical case (\cite {Zyg}, ch. 12.4); one only needs to apply \pro{P1} at certain steps of the proof.
\begin{OldTheorem}[Marcinkiewicz interpolation theorem]\label{M} If a subadditive  operator $T$ satisfies the weak-$L^{p_1}$ and the weak-$L^{p_2}$ estimates ($1\le p_1<p_2\le \infty$), then the operator $T$ satisfies the strong-$L^p$ bound for any $p$, satisfying $p_1<p<p_2$, .
\end{OldTheorem}
}
We will use below weak-$L^1$ bound of the standard maximal function
\begin{equation}\label{1-1}
	\MM f(x)=\sup_{B\in \ZB:\, x\in B}\langle f\rangle_B=\sup_{B\in \ZB:\, x\in B} \frac{1}{\mu(B)}\int_B|f|
\end{equation}
defined on a measure space with a ball-basis $\ZB$.
\begin{OldTheorem}[\cite{Kar3}, Theorem 4.1]\label{T1-1}
	The maximal function \e {1-1} satisfies weak-$L^1$ inequality. Namely, 
	\begin{equation}\label{h38}
		\mu^*\{x\in X:\,\MM f(x)>\lambda\}|\lesssim \|f\|_{L^1(X)}/\lambda,\quad \lambda>0.
	\end{equation} 
\end{OldTheorem}
\section{Preliminary properties of ball bases}\label{S2}
Some of lemmas proved in this section are versions of similar statements from \cite{Kar3, Kar1}. We find necessary to provide their proofs though.  Let $\ZB$ be a ball-basis in the measure space $(X,\mu)$. From B4) condition it follows that if balls $A,B$ satisfy $A\cap B\neq\varnothing$, $\mu(A)\le 2\mu(B)$, then $A\subset B^*$. This property will be called two balls relation. We say a set $E\subset X$ is bounded if $E\subset B$ for a ball $B\in\ZB$.
\begin{lemma}\label{L0}
	If $(X,\mu)$ is a measure space equipped with a ball-basis $\ZB$ and $\mu(X)<\infty$, then $X\in \ZB$.
\end{lemma}
\begin{proof}
	Fix an arbitrary point $x_0\in X$ and let $\ZA\subset \ZB$ be the family of balls containing $x_0$. By B2) condition we can write $X=\cup_{A\in \ZA}A$. Then there exists a $B\in \ZA$ such that $\mu(B)>\frac{1}{2} \sup_{A\in \ZA}\mu(A)$ and using B4) condition we obtain
	\begin{equation*}
		X=\cup_{A\in \ZA}A\subset B^*.
	\end{equation*}
	Thus $X=B^*\in \ZB$.
\end{proof}
\begin{lemma}\label{L5}
	Let $(X,\mu)$ be a measure space equipped with a ball-basis $\ZB$ and $G\in \ZB$. Then there exists a sequence of balls $G=G_1,G_2, \ldots$ (finite or infinite) such that 
	\begin{equation}
		X=\cup_kG_k,\quad G_k^*\subset G_{k+1}. 
	\end{equation}
	Moreover for any $B\in \ZB$ there is a ball $G_n\supset B$.  If the ball-basis is doubling, then we can additionally claim 
	\begin{equation}\label{x71}
		2\mu(G_k)\le \mu(G_{k+1})\le \gamma\mu(G_k), 
	\end{equation}
	with an admissible constant $\gamma>2$. 
\end{lemma}

\begin{proof}
	Choose an arbitrary point $x_0\in G$ and let $\ZA\subset \ZB$ be the family of balls containing the point $x_0$. Choose a sequence $\eta_n\nearrow\eta=\sup_{A\in \ZA}\mu(A)$, such that $\eta_1<\mu(G)$, where $\eta$ can also be infinity. Let us see by induction that there is an increasing  sequence of balls $G_n\in \ZA$ such that $G_1=G$, $\mu(G_n)> \eta_n$ and $G_n^*\subset G_{n+1}$. The base of induction is trivial. Suppose we have already chosen the first balls $G_k$, $k=1,2,\ldots, l$. There is a ball $A\in \ZA$ so that $\mu(A)>\eta_{l+1}$. Let $C$ be the biggest in measure among the balls $A$ and $G_l^*$ and denote $G_{l+1}=C^*$. By property B4) we get $A\cup G_l^*\subset C^*=G_{l+1}$. This implies $\mu(G_{l+1})\ge \mu(A)>\eta_{l+1}$ and $G_{l+1}\supset G_l^*$, completing the induction.
	Let us see that $G_n$ is our desired sequence of balls. Indeed, let $B$ be an arbitrary ball. By B2) property there is a ball $A$ containing $x_0$, such that $A\cap B\neq \varnothing$. Then by property B4) we may find a ball $C\supset A\cup B$. For some $n$ we will have $\mu(C)\le 2\mu(G_n)$ and so once again using B4), we get $B\subset  C\subset G_{n}^*\subset G_{n+1}$. 
	
	The second part of the lemma can be proved by a similar argument. Using the doubling condition (\df {DD}) one can easily find a sequence of balls $G=G_1,G_2, \ldots$ such that
	\begin{equation}
		G_k^*\subset G_{k+1},\quad 2\mu(G_k^*)\le \mu(G_{k+1})\le \eta \mu(G_k^*)\le \eta\ZK \mu(G_k).
	\end{equation}
	If $\mu(X)<\infty$, then clearly the ball sequence will be finite and its last term say $G_n$ will coincide with $X$. So the proof follows. If $\mu(X)=\infty$, then our ball sequence is infinite and $\mu(G_n)\to \infty$. As in the first part of the proof, for an arbitrary ball $B$ we can find a ball $C\supset B$ such that $C\cap G\neq \varnothing$. There is a ball $G_n$ such that $\mu(G_n)> \mu(C)$ and we get $B\subset  C\subset G_{n}^*\subset G_{n+1}$. This completes the proof of lemma.
\end{proof}
\begin{remark}\label{R2}
	If $\mu(X)<\infty$, then the sequence $G_n$ in \lem{L5} is finite and its last-most term coincides with $X$.
\end{remark}

\begin{lemma}\label{L1-1}
	Let $(X,\mu)$ be a measure space with a ball bases $\ZB$. If $E\subset X$ is bounded and a family of balls $\ZG $ is a covering of $E$, i.e. $E\subset \bigcup_{G\in \ZG}G$, then there exists a finite or infinite sequence of pairwise disjoint balls $G_k\in \ZG$ such that $E \subset \bigcup_k G_k^{*}$.
\end{lemma}
\begin{proof}
	The boundedness of $E$ implies $E\subset B$ for some $B\in \ZB$. If there is a ball $G\in\ZG$ so that $G\cap B\neq \varnothing$, $\mu(G)> \mu(B)$, then by two balls relation we will have $E\subset B\subset G^{*}$. Thus our desired sequence can be formed by a single element $G$. Hence we can suppose that every $G\in\ZG$ satisfies $G\cap B\neq \varnothing$, $\mu(G)\le \mu(B)$ and again by two balls relation $G\subset B^*$. Therefore, $\bigcup_{G\in\ZG}G\subset B^{*}$.
	Choose $G_1\in \ZG$, satisfying $\mu(G_1)> \frac{1}{2}\sup_{G\in\ZG}\mu(G)$. Then, suppose by induction we have already chosen elements $G_1,\ldots,G_k$ from $\ZG$. Choose $G_{k+1}\in \ZG$  disjoint with the balls $G_1,\ldots,G_k$ such that
	\begin{equation}\label{b1}
		\mu(G_{k+1})> \frac{1}{2}\sup_{G\in \ZG:\,G\cap G_j=\varnothing,\,j=1,\ldots,k}\mu(G).
	\end{equation}
	If for some $n$ we will not be able to determine $G_{n+1}$ the process will stop and we will get a finite sequence $G_1,G_2,\ldots, G_n$. Otherwise our sequence will be infinite. We shall consider the infinite case of the sequence (the finite case can be done similarly). Since the balls $G_n$ are pairwise disjoint and $G_n\subset B^{*}$, we have $\mu(G_n)\to 0$. Choose an arbitrary $G\in\ZG$ with $G\neq G_k$, $k=1,2,\ldots $ and let $m$ be the smallest integer satisfying
	$\mu(G)\ge 2\mu(G_{m+1})$.
	So we have $G \cap G_j\neq\varnothing$
	for some $1\le j\le m$, since otherwise by \e{b1}, $G$ had to be chosen instead of $G_{m+1}$. Besides, we have $\mu(G)< 2\mu(G_{j})$ because of the minimality property of $m$, and so by two balls relation $G\subset G_{j}^{*}$. Since $G\in \ZG$ was chosen arbitrarily, we get $E \subset\bigcup_{G\in \ZG} G\subset \bigcup_k G_k^{*}$.
\end{proof}

\begin{definition}\label{D1}
	For a measurable set $E\subset X$ a point $x\in E$ is said to be density point if for any $\varepsilon>0$ there exists a ball $B\ni x$ such that
	\begin{equation*}
		\mu(B\cap E)>(1-\varepsilon )\mu(B).
	\end{equation*} 
	We say a ball basis satisfies the density property if for any measurable set $E$ almost all points $x\in E$ are density points. 
\end{definition}
\begin{lemma}\label{L12}
	Every ball basis $\ZB$ satisfies the density condition.
\end{lemma}
\begin{proof}
	Applying \lem{L5} one can check that it is enough to establish the density property for the bounded measurable sets. Suppose to the contrary there exist a bounded measurable set $E\subset X$ together its subset $F\subset E$ (maybe non-measurable) with an outer measure $\mu^*(F)>0$
	such that
	\begin{equation}\label{z38}
		\mu(B\setminus E)>\alpha\mu(B) \text{ whenever }B\in\ZB,\, B\cap F\neq\varnothing,
	\end{equation}
	where $0<\alpha<1$. According to the definition of the outer measure one can find a measurable set $\bar F$ such that 
	\begin{equation}\label{z2}
		F\subset \bar F\subset E, \quad \mu(\bar F)=\mu^*(F).
	\end{equation}
	By B3)-condition there is a sequence of balls $B_k$, $k=1,2,\ldots $, such that
	\begin{equation}\label{z42}	
		\mu\left(\bar F\bigtriangleup (\cup_k B_k)\right)<\varepsilon.
	\end{equation}
	We discuss two collections of balls $B_k$, satisfying either $B_k\cap F=\varnothing$ or $B_k\cap F\neq\varnothing$. For the first collection we have 
	\begin{equation*}
		\mu^*(F)\le \mu\left(\bar F\setminus \bigcup_{k:\,B_k\cap F=\varnothing} B_k\right)\le \mu(\bar F)=\mu^*(F).
	\end{equation*}
	This implies 
	\begin{equation*}
		\mu\left(\bar F\bigcap \left(\bigcup_{k:\,B_k\cap F=\varnothing} B_k\right)\right)=0
	\end{equation*}
	and therefore, combining also \e{z42}, we obtain
	\begin{equation}\label{z3}
		\mu\left(\bigcup_{k:\,B_k\cap F=\varnothing} B_k\right)<\varepsilon.
	\end{equation}
	Now consider the balls, satisfying $B_k\cap F\neq\varnothing$. It follows from \e {z38} and \e{z2} that 
	\begin{equation}\label{z41}
		\mu(B_k\setminus \bar F)\ge \mu(B_k\setminus E)>\alpha \mu(B_k),\quad k=1,2,\ldots.
	\end{equation}
	Since $E$ and so $\bar F$ are bounded, applying \lem {L1-1} and \e{z42}, one can find a subsequence of pairwise disjoint balls $\tilde B_k$, $k=1,2,\ldots$, such that
	\begin{equation*}
		\mu\left(\bar F\setminus\cup_k\tilde B_k^{*}\right)<\varepsilon.
	\end{equation*} 
	Thus, from B4)-condition, \e{z42}, \e {z3} and \e {z41}, we obtain
	\begin{align*}
		\mu^*(F)< \mu(\bar F)&\le \mu\left(\cup_k\tilde B_k^{*}\right)+\varepsilon\le \ZK\sum_k\mu(\tilde B_k)+\varepsilon\\
		&= \ZK\mu\left(\bigcup_{k:\,\tilde B_k\cap F=\varnothing} \tilde B_k\right)+\ZK\sum_{k:\,\tilde B_k\cap F\neq\varnothing} \mu(\tilde B_k)+\varepsilon\\
		&<\ZK\varepsilon+\frac{\ZK}{\alpha}\sum_k\mu(\tilde B_k\setminus \bar F)+\varepsilon\\
		&\le (\ZK+1)\varepsilon+\frac{\ZK}{\alpha}\mu\left(\bar F\bigtriangleup( \cup_k B_k)\right)<\varepsilon\left(2\ZK+1+\frac{\ZK}{\alpha}\right).
	\end{align*}
	Since $\varepsilon $ can be arbitrarily small, we get $\mu^*(F)=0$ and so a contradiction.
\end{proof}

\section{Proof of \trm{T5}}

The following two lemmas are standard and well-known in the classical situations.
\begin{lemma}\label{L6}
	For any function $f\in L^r_{loc}(X)$ and balls $A, B$ with $A\cap B\neq\varnothing$, $\mu(A)\le \mu(B)$ we have
	\begin{equation}\label{x44}
		|f_A-f_B|\lesssim\left(\frac{\mu(B)}{\mu(A)}\right)^{1/r}\cdot  \langle f\rangle_{\#,A}^*.
	\end{equation}
	\begin{proof}
		First suppose that $A\subset B$. Then we get
		\begin{align}
			\left|f_{A}-f_{B}\right|&\le \left(\frac{1}{\mu(A)}\int_{A}|f-f_{B}|^r\right)^{1/r}\\
			&\le\left(\frac{\mu(B)}{\mu(A)}\right)^{1/r} \cdot \left(\frac{1}{\mu(B)}\int_{B}|f-f_{B}|^r\right)^{1/r}\le \left(\frac{\mu(B)}{\mu(A)}\right)^{1/r}\cdot \langle f\rangle_{\#,B}.\label{x72}
		\end{align}
	In the general case, we can write $A\subset B^*$ and so applying \e{x72} we obtain
	\begin{align}
	\left|f_{A}-f_{B}\right|&\le \left|f_{A}-f_{B^*}\right|+\left|f_{B}-f_{B^*}\right|\\
	&\le \left(\left(\frac{\mu(B^*)}{\mu(A)}\right)^{1/r}+\left(\frac{\mu(B^*)}{\mu(B)}\right)^{1/r}\right) \langle f\rangle_{\#,B^*}\lesssim \left(\frac{\mu(B)}{\mu(A)}\right)^{1/r}\cdot  \langle f\rangle_{\#,A}^*.
	\end{align}
	\end{proof}
\end{lemma}
\begin{lemma}\label{L10}
	Let $\ZB$ be a doubling ball-basis on $X$. Then for any $f\in L^r_{loc}(X)$ and balls $A, B$ with $A\cap B\neq\varnothing$, $\mu(A)\le \mu(B)$ we have that
	\begin{equation}\label{x42}
		\langle f-f_A\rangle_B\lesssim (1+\log(\mu(B)/\mu(A)))\cdot \langle f\rangle_{\#,A}^*.
	\end{equation}
\end{lemma}
\begin{proof} 
Applying the second part of \lem{L5}, we find a sequence of balls $A=G_0\subset G_1\subset \dots$, satisfying \e{x71}. We have
$\mu(G_n)<\mu(B^*)\le \mu(G_{n+1})$ for some integer $n$, where we can write $n\lesssim 1+\log(\mu(B)/\mu(A))$.
	By \lem {L6}, 
	\begin{align}
		&|f_{G_k}-f_{G_{k+1}}|\lesssim \langle f\rangle_{G_{k},\#}^*\lesssim \langle f\rangle_{\#,A}^* \hbox{ (see } \e{x5}),\\
		&|f_{G_n}-f_{B^*}|\lesssim \langle f\rangle_{G_n,\#}^*\lesssim\langle f\rangle_{\#,A}^*,
	\end{align}
	so we obtain
	\begin{align*}
		\langle f-f_A\rangle_B&\le \langle f-f_{B^*}\rangle_B+|f_A-f_{B^*}|\\
		&\lesssim \langle f-f_{B^*}\rangle_{B^*}+ \sum_{k=0}^{n-1}|f_{G_k}-f_{G_{k+1}}|+|f_{G_n}-f_{B^*}|\\
		&\lesssim n\langle f\rangle_{\#,A}^*\lesssim (1+\log(\mu(B)/\mu(A)))\langle f\rangle_{\#,A}^*.
	\end{align*}
\end{proof}

\begin{proof}[Proof of \trm{T5}]
		First observe that we need to consider only the operator $T^+$, since if $T_\alpha$ is vanishing, then so are operators $-T_\alpha$ and $|T_\alpha|$. Hence, we suppose 
	\begin{equation}\label{r42}
		Tf(x)=T^+f(x)=\sup_\alpha T_\alpha f(x).
	\end{equation}
	Let $f\in L^r(X)$ and $B\in \ZB$. Applying the vanishing property (see \df{D2}), we find a ball $B'\supset B^*$ such that 
	\begin{equation}
		\sup_\alpha \OSC_B(T_\alpha(\ZI_{B''}))<\delta
	\end{equation}
	for every ball $B''\supset B'$. Then let $G_n$ be the sequence of balls generated from \lem{L5}. We have $\langle f\rangle_{G_n}^*\to 0$ and so there is a ball $G_n\supset B'\supset B^*$ such that $\langle f\rangle^*_{G_n} <\delta$.
	Thus for the ball $G=G_n^*$ we have
	\begin{align}
		&B^*\subset G,\\
		&\sup_\alpha \OSC_B(T_\alpha(\ZI_{G}))<\delta,\label{y86}\\
		&\sup_\alpha\OSC_{B}(T_\alpha(f\cdot \ZI_{X\setminus G}))\le\sup_\alpha\OSC_{G_n}(T_\alpha(f\cdot \ZI_{X\setminus G_n^*}))\\
		&\qquad\qquad\qquad\qquad\qquad \,\,\,\,\le  \sup_\alpha \ZL_\omega(T_\alpha)\cdot \langle f\rangle^*_{G_n}<\delta \sup_\alpha\ZL_\omega(T_\alpha). \label{y85}
	\end{align}
	In the last bound we use \e{y83}, omitting the factor, containing $\omega$ in it (the factor $\omega$ will be used later). Fixing the ball $G$, we split the function $f$ as
	\begin{align*}
		f&=f\cdot \ZI_{X\setminus G}+(f-f_B)\ZI_{B^*}+(f-f_B)\ZI_{G\setminus B^*}+f_B\ZI_{G}\\
		&=f_0+f_1+f_2+f_3.
	\end{align*}
	Choosing small enough $\delta>0$, from \e{y86}, \e{y85} we may obtain
	\begin{equation}\label{x2}
		\sup_\alpha \OSC_B(T_\alpha(f_0))<\varepsilon,\quad \sup_\alpha \OSC_B(T_\alpha(f_3))<\varepsilon
	\end{equation}
	for an arbitrary $\varepsilon>0$. Consider the set $E_{B}=\{y\in B:\, |Tf_1(y)|\le \lambda\}$, which can be non-measurable. Since $\|T\|_{L^r(X)\to L^{r,\infty}(X)}<\infty$ for 
	\begin{align}
		\lambda=\ZK^{1/r}(1-\beta)^{-1/r}\|T\|_{L^r(X)\to L^{r,\infty}(X)}\langle f_1\rangle_{B^*}
	\end{align}
	we have
	\begin{equation}
		\mu^*(B\setminus E_{B})=\mu^*\{y\in B:\, |Tf_1(y)|>\lambda\}\le\frac{(1-\beta )\mu(B^*)}{\ZK}\le  (1-\beta )\mu(B).
	\end{equation}
	Thus, we may find a measurable set $\bar E_B\subset E_B$ such that 
\begin{equation}\label{r51}
		\mu (\bar E_B)\ge \beta\mu(B)
\end{equation}
	Then, applying \lem{L10}, we can write
	\begin{align}
		& |Tf_1(y)|\lesssim(1-\beta)^{-1/r}\|T\|_{L^r(X)\to L^{r,\infty}(X)}\langle f-f_B\rangle_{B^*}\\
		&\qquad \quad\lesssim(1-\beta)^{-1/r} \|T\|_{L^r(X)\to L^{r,\infty}(X)}\langle f\rangle^*_{\#,B}\text{ for every }y\in \bar E_B.\label{y87}
	\end{align}
	Now choose two arbitrary points $x,x'\in \bar E_B$ and suppose $Tf(x)>Tf(x')$. Clearly, for some $\alpha$ we have
	\begin{align}
		|Tf(x)-Tf(x')|=Tf(x)-Tf(x')\le 2( T_\alpha f(x)-T_\alpha f(x')).
	\end{align}
	On the other hand, using \e{y87}, for small enough $\varepsilon$ in \e{x2} we obtain
	\begin{align}
		T_\alpha f(x)-&T_\alpha f(x')\\
		&\le \OSC_B(T_\alpha(f_0))+\OSC_B(T_\alpha(f_3))\\
		&\qquad +|T_\alpha f_1(x)|+|T_\alpha f_1(x')|+|T_\alpha f_2(x)-T_\alpha f_2(x')|\\
		&\le 2\varepsilon +|T f_1(x)|+|T f_1(x')|+|T_\alpha f_2(x)-T_\alpha f_2(x')|\\
		&\lesssim (1-\beta)^{-1/r}\|T\|_{L^r(X)\to L^{r,\infty}(X)}\langle f\rangle^*_{\#,B}\\
		&\qquad +|T_\alpha f_2(x)-T_\alpha f_2(x')|.\label{y28}
	\end{align}
	Applying \e{y83} with $\omega(t)=\log^{-1}(1+t)$, we get
	\begin{align}
		|T_\alpha f_2(x)-T_\alpha f_2(x')|&\le \OSC_{B}(T_\alpha (f_2))\\
		&\le \ZL_\omega(T_\alpha)\sup_{C\in \ZB,\,C\supset B}\bigg(\log^{-1}\left(1+\frac{\mu(C)}{\mu(B)}\right)\langle f_2\rangle_{C}\bigg)\\
		&\le 2\ZL_\omega(T_\alpha)\log^{-1}\left(1+\frac{\mu(A)}{\mu(B)}\right)\langle f_2\rangle_{A}
	\end{align} 
	for some ball $A\supset B$. Then, applying \lem{L10}, we obtain
	\begin{align}
		|T_\alpha f_2(x)-T_\alpha f_2(x')|&\le 2\ZL_\omega(T_\alpha)\log^{-1}\left(1+\frac{\mu(A)}{\mu(B)}\right)\langle f-f_B\rangle_{A}\\
		&\lesssim \ZL_\omega(T_\alpha)\log^{-1}\left(1+\frac{\mu(A)}{\mu(B)}\right)\log\left(1+\frac{\mu(A)}{\mu(B)}\right)\langle f\rangle^*_{\#,B}\\
		&\le \sup_\alpha \ZL_\omega(T_\alpha)\langle f\rangle^*_{\#,B}.\label{y29}
	\end{align}
	Hence, combining \e{r51}, \e{y28} and \e{y29}, we obtain \e{y30}.
\end{proof}
\begin{proof}[Proof of \trm{T1}]
The proof of theorem is similar that of \trm{T5}, so we omit its details.  First, we can  consider only the operator $T^+$. Thus we have \e{r42}. For a given ball $B$ we consider the set 
\begin{equation*}
	E_{B}=\{y\in B:\, |T(f\cdot \ZI_{B^*})(y)|\le \lambda\}
\end{equation*}
where
\begin{align}
	\lambda=\ZK^{1/r}(1-\beta)^{-1/r}\|T\|_{L^r(X)\to L^{r,\infty}(X)}\langle f\rangle_{B^*}.
\end{align}
For a measurable set $\bar E_B\subset E_B$ we will have
\begin{align}
	&\mu(\bar E_{B})>\beta\mu(B),\\
	& |T(f\cdot \ZI_{B^*})(y)|\lesssim(1-\beta)^{-1/r}\|T\|_{L^r(X)\to L^{r,\infty}(X)}\langle f\rangle_{B^*}\\
	&\qquad \quad\le(1-\beta)^{-1/r} \|T\|_{L^r(X)\to L^{r,\infty}(X)}\| f\|_\infty\text{ for every }y\in \bar E_B.\label{y88}
\end{align}
Choose arbitrary points $x,x'\in \bar E_B$ and suppose $Tf(x)>Tf(x')$. Then for some $\alpha$ we have
\begin{align}
	|Tf(x)-Tf(x')|=Tf(x)-Tf(x')\le 2( T_\alpha f(x)-T_\alpha f(x'))
\end{align}
and hence, applying \e{y83} with $\omega(t)\equiv 1$ and \e{y88}, we obtain
\begin{align*}
	T_\alpha f(x)-&T_\alpha f(x')\\
	&\le|T_\alpha (f\cdot \ZI_{B^*})(x)|+|T_\alpha  (f\cdot \ZI_{B^*})(x')|+  \OSC_B(T_\alpha (f\cdot \ZI_{X\setminus B^*}))\\
	&\lesssim (1-\beta)^{-1/r} \|T \|_{L^r(X)\to L^{r,\infty}(X)}\| f\|_\infty+\ZL(T_\alpha ) \cdot \|f\|_\infty\\
	&\le ((1-\beta)^{-1/r} \|T\|_{L^r(X)\to L^{r,\infty}(X)}+\sup_\alpha \ZL(T_\alpha )) \cdot \|f\|_\infty,
\end{align*}
completing the proof of the theorem.
\end{proof}

\section{Proof of \trm{T11}}\label{S5}
{\begin{lemma}\label{L9}
	For any measurable finite function $f$ on $X$ and for any ball $B$ there exists a measurable set $E\subset B$ such that
	\begin{align*}
		\mu(E)\ge \alpha \mu(B),\quad \OSC_E(f)\le \OSC_\alpha(f).
	\end{align*}
\end{lemma}
\begin{proof}
	Recall that the functional $\OSC_{B,\alpha}(f)$ can be equivalently defined as follows: that is $\OSC_{B,\alpha}(f)=\inf_{a<b}(b-a)$, where the infimum is taken over all the numbers $a<b$ satisfying $\mu\{x\in B:\, f(x)\in [a,b]\}>\alpha\mu(B)$. Hence, there is a sequence of intervals $[a_n,b_n]$ such that $b_n-a_n\to \OSC_\alpha(f)$ and $\mu\{x\in B:\, f(x)\in [a_n,b_n]\}>\alpha\mu(B)$. It is clear that $a_n$ and $b_n$ are bounded sequences and so those have convergent subsequences, $a_{n_k}\to a$ and $b_{n_k}\to b$. One can check that the set $E=\{x\in B:\, f(x)\in [a,b]\}$ satisfies the required conditions.
\end{proof}}
{Let $\gamma$, $\theta$ be the constants from \lem{L5} (see \e{x71}) and from  \df{D3} respectively. We have $\gamma>2$, $0<\theta<1$ and the ball basis constant $\ZK>1$. Thus we can write
\begin{equation}\label{x74}
	\alpha= 1-\frac{\varepsilon \theta}{4\gamma^2\ZK}>1/2
\end{equation}
for any $0<\varepsilon<1$.}
\begin{proposition}\label{P}
	Let $\ZB$ satisfy doubling and regularity properties,  $0<\varepsilon<1$ and suppose that $\alpha$ is the constant in \e{x74}.
	If $g\in \BMO_\alpha(X)$ (perhaps non-measurable), then
	\begin{equation}\label{x10}
		\mu^*\{x\in B:\, |g(x)|>\lambda+\|g\|_{\BMO_\alpha}\}\le \varepsilon\cdot \mu^*\{x\in B:\, |g(x)|>\lambda\}
	\end{equation}
	for any ball $B\in \ZB$ and any number $\lambda>0$, provided
	\begin{equation}\label{x17}
		\mu^*\{x\in B:\, |g(x)|>\lambda\}\le \frac{\theta }{5\ZK\gamma^2}\cdot \mu(B).
	\end{equation}
\end{proposition}
\begin{proof}
	Fix a ball $B$ and a number $\lambda>0$, satisfying  \e{x17}. Then consider the set $E(\lambda)=\{x\in B:\,|g(x)|>\lambda\}$, which can be non-measurable. Choose a measurable set $F$ such that
	\begin{equation}\label{x50}
		B\supset F\supset E(\lambda),\quad \mu(F)=\mu^*(E(\lambda))\le \frac{\theta }{5\ZK\gamma^2}\cdot \mu(B) \,(\text{see } \e{x17}).
	\end{equation}
	Applying the density property (\lem{L12}) for any point $x\in F$
	we find a ball $G_0(x)\ni x$ such that $\mu(G_0(x)\cap F)\ge\alpha\theta \mu(G_0(x))/(2\gamma)$. Then, applying the second part of \lem{L5}, one can find a sequence of balls $G_0=G_0(x)\subset G_1\subset G_2\subset\ldots $, satisfying the conditions of the lemma.  {Note that this sequence of balls can be either infinite or finite. We claim that there is a biggest integer $n$ such that 
	\begin{equation}\label{x51}
		\frac{\mu(G_{n-1}\cap F)}{\mu(G_{n-1})}\ge\frac{ \alpha\theta}{2\gamma}
	\end{equation}
	and that $G_{n-1}$ is not the last term of our sequence. Indeed, the above inequality holds if $n-1=0$. Hence it is enough to show that we have reverse inequality for a nonempty subfamily of balls $G_{n}$ with $n>n_0$. Such a statement is clear if our sequence is infinite, since in that case, according to \lem{L5}, we will have 
	\begin{equation*}
		\lim_{n\to\infty}\frac{\mu(G_{n-1}\cap F)}{\mu(G_{n-1})}=0.
	\end{equation*}
	If the sequence is finite and $G_m$ is its last term, then $G_m=X$ (see \rem{R2}). Thus, using \e{x50} and \e{x74}, we obtain a reverse inequality
	\begin{equation}
		\frac{\mu(G_{m}\cap F)}{\mu(G_{m})}= \frac{\mu(F)}{\mu(X)}\le \frac{\theta }{5\ZK\gamma^2}<\frac{ \alpha\theta}{2\gamma}.
	\end{equation}
	Hence we can suppose that there is a biggest integer $n$ for which \e{x51} holds and the ball $G(x)=G_n$ satisfies
	\begin{equation}\label{x8}
		\frac{\alpha \theta}{2\gamma^2}\le \frac{\mu(G(x)\cap F)}{\mu(G(x))}<\frac{\alpha\theta}{2\gamma}\, 
	\end{equation}
where the left hand side inequality follows from \e{x71} and \e{x51}. } Since  $\mu(G_{n+1})\le \gamma \mu(G_n^*)$ and $G^*(x)=G_n^*\subset G_{n+1}$ we obtain
	\begin{equation}\label{x73}
		\frac{\mu(G^*(x)\cap F)}{\mu(G^*(x))}\le  \frac{\gamma\mu(G_{n+1}\cap F)}{ \mu(G_{n+1})}< \frac{\alpha\theta}{2}.
	\end{equation}
	The left hand side inequality in \e{x8} together with \e{x74} and \e{x50} imply
	\begin{equation}\label{x18}
		\mu(G^*(x))\le \ZK\mu(G(x))\le \frac{2\ZK\gamma^2}{\alpha\theta}\cdot \mu(F)<  \mu(B).
	\end{equation}
	Applying \lem{L1-1}, then we find a sequence of pairwise disjoint balls $\{B_k\}\subset \{G(x):\, x\in F\}$ such that
	\begin{equation}\label{x11}
		F\subset \cup_kB_k^*.
	\end{equation}
	By \lem{L9} there are sets $E_k\subset B^*_k$ such that
	\begin{align}
		&\mu(E_k)\ge \alpha\mu(B^*_k),\label{x7}\\
		&\OSC_{E_k}(g)\le \OSC_{B_k^*,\alpha}(g)\le  \|g\|_{\BMO_\alpha}.\label{x12}
	\end{align}
	Applying the regularity property (see \df{D3}) and \e{x18}, we can write $\mu(B_k^*\cap B)\ge \theta \mu(B_k^*)$. Then from \e{x73} and \e{x7} we obtain
	\begin{equation}\label{x15}
		\mu(B_k^*\cap F)<\frac{\alpha\theta\mu(B_k^*)}{2}\le \frac{\alpha}{2} \mu(B_k^*\cap B)<  \frac{1}{2} \mu(B_k^*\cap B)
	\end{equation}
	and
	\begin{align}
		\mu(E_k\cap B)&\ge  \mu(B^*_k\cap B)-\mu(B_k^*\setminus E_k)\\
		&\ge \mu(B^*_k\cap B)-(1-\alpha)\mu(B_k^*)\\
		&\ge \mu(B_k^*\cap B)-\frac{1-\alpha}{\theta}\cdot \mu(B_k^*\cap B)\\
		&\ge  \frac{1}{2}  \mu(B_k^*\cap B),\label{x14}
	\end{align}
	respectively. { In the last estimate we used the inequality $1-(1-\alpha)/\theta>1/2$, which follows directly by substituting the value of $\alpha$ from \e{x74} and taking into account the bounds of the involved parameters.} Since both $B_k^*\cap F$ and $E_k\cap B$ are subsets of the set $B_k^*\cap B$, from \e{x14} and \e{x15} we obtain
	$\mu((E_k\cap B)\setminus F)=\mu(E_k\cap B)-\mu(B_k^*\cap F)>0$.
	Thus we can find  a point $x_k\in E_k\setminus F$ so that $|g(x_k)|\le\lambda$. Then by \e{x12} for any $x\in E_k$ we will have $|g(x)-g(x_k)|\le \|g\|_{\BMO_\alpha}$ and so $|g(x)|\le \lambda+\|g\|_{\BMO_\alpha}$. This implies 
	\begin{align}
		\mu^* (E\big(\lambda+\|g\|_{\BMO_\alpha})\big)&\le \sum_k\mu(B_k^*\setminus E_k)&\quad (\hbox{see }\e{x50},\e{x11})\\
		&\le (1-\alpha) \sum_k\mu(B_k^*)&\quad (\hbox{see }\e{x7})\\
		&\le \ZK(1-\alpha) \sum_k\mu(B_k)&\\
		&\le  \frac{2\gamma^2 \ZK}{\alpha\theta}(1-\alpha) \sum_k\mu(B_k\cap F)&\quad (\hbox{see }\e{x8})\\
		&\le\varepsilon\cdot \mu(F)\\
		&= \varepsilon\cdot \mu^*(E(\lambda)).&\quad (\hbox{see }\e{x50})
	\end{align}
	Hence we get \e{x10}.
\end{proof}

\begin{proof}[Proof of \trm{T11}]
	The inclusion $ \BMO(X)\subset \BMO_\alpha(X)$ is immediate. To prove the converse embedding we let $\alpha$ be the number from \e{x74} and $\varepsilon=1/2$.  By \lem{L9} any ball $B$ contains a measurable set $E\subset B$ such that
	\begin{align}
		&\mu(E)\ge \alpha\mu(B),\\
		&\OSC_E(f)\le \OSC_{B,\alpha}(f).
	\end{align}
Thus for $a_B=(\SUP_E(f)+\INF_E(f))/2$ we have
\begin{align}
	\mu^*\{x\in B:\, &|f(x)-a_B|>\OSC_{B,\alpha}(f)\}\\
	&\le \mu^*\{x\in B:\, |f(x)-a_B|>\OSC_{E}(f)\}\\
	&\le\mu(B\setminus E) \le (1-\alpha)\mu(B),
\end{align}
and therefore
	\begin{align}
		\mu^*\{x\in B:\, &|f(x)-a_B|>\|f-a_B\|_{\BMO_\alpha}\}\\
		&=\mu^*\{x\in B:\, |f(x)-a_B|>\|f\|_{\BMO_\alpha}\}\\
		&\le  \mu^*\{x\in B:\, |f(x)-a_B|>\OSC_{B,\alpha}(f)\}\\
		&\le (1-\alpha)\mu(B)\\
		&=\frac{ \theta }{8\ZK\gamma^2}\cdot \mu(B).
	\end{align}
{Using this we can apply \pro{P} with $\varepsilon=1/2$ and $\lambda=(k-1)\|f\|_{\BMO_\alpha}$ for all integers $k\ge 2$. Hence we can write
\begin{align}
	\mu^*\{x\in B:\, |f(x)-a_B|>&k\|f\|_{\BMO_\alpha}\}\\
	&\le \frac{1}{2}\cdot\mu^*\{x\in B:\, |f(x)-a_B|>(k-1)\|f\|_{\BMO_\alpha}\}.
\end{align}
Applying this inequality consecutively for integers $k=2,3,\ldots$, we obtain the bound
\begin{equation*}
	\mu^*\{x\in B:\, |f(x)-a_B|>n\|f\|_{\BMO_\alpha}\}\le \left(\frac{1}{2}\right)^{n-1}\cdot \mu(B).
\end{equation*}
This easily implies
	\begin{equation}\label{r53}
		\frac{1}{\mu(B)}\cdot \mu^*\{x\in B:\, |f(x)-a_B|>t\}\le c_1\exp \left(\frac{-c_2t}{\|f\|_{\BMO_\alpha}}\right)
	\end{equation}
for all $t>0$, with admissible constants $c_1$ and $c_2$. }
Note that the integral of a positive non-measurable function is defined by the distribution function (see \e{r52}). Finally, from \e{r53} it follows that 
\begin{equation*}
	\langle f\rangle_{\#, B}\le 2\left(\frac{1}{\mu(B)}\int_B|f-a_B|^r\right)^{1/r}\lesssim \|f\|_{\BMO_\alpha}
\end{equation*}
for any ball $B$, and so we get $\|f\|_\BMO\lesssim 	\|f\|_{\BMO_\alpha}$. This completes the proof of \trm{T11}.
\end{proof}

\section{Examples of $\BO$ operators on BMO}
\subsection{Calder\'on-Zygmund operators} 
A Calder\'on-Zygmund operator $T$ is a linear operator such that
\begin{equation}\label{CZ}
	Tf(x)=\int_{\ZR^d}K(x,y)f(y)dy,\quad x\notin \supp f,
\end{equation}
for any compactly supported continuous function $f\in C(\ZR^d)$, where the kernel $K:\ZR^d\times \ZR^d\to \ZR$ satisfies the conditions
\begin{align}
	&|K(x,y)|\le \frac{C}{|x-y|^d}\text{ if }x\neq y,\label{x24}\\
	&|K(x,y)-K(x',y)|+|K(y,x)-K(y,x')|\le \omega\left(\frac{|x-x'|}{|x-y|}\right)\cdot \frac{1}{|x-y|^d},\\
	&\text{ whenever } |x-y|>2|x-x'|.\label{x9}
\end{align}
Here $\omega:(0,1)\to (0,1)$ is a non-decreasing function, satisfying $\omega(2x)\le c\cdot \omega(x)$ (called modulus of continuity) and the Dini condition $\int_0^1\omega(t)/tdt<\infty$. Also it is supposed that $T$ has a bounded extension on $L^2(\ZR^d)$. 
\begin{lemma}\label{L1}
	If the function $\omega$ satisfies the logarithmic Dini condition
	\begin{equation}\label{x40}
		\int_0^1\frac{\omega(t)\log(1/t)}{t}dt<\infty,
	\end{equation} 
	then the operator \e{CZ}  satisfies \e{y83}.
\end{lemma}
\begin{proof}
	Let $f\in L^r(\ZR^d)$ and $B=B(x_0,R)\subset \ZR^d$ be the Euclidean open ball of a center $x_0$ and a radius $R$. One can choose $B^*=B(x_0,R^*)$ as a hull ball of $B$, where $R^*=(1+2\cdot 2^{1/d})R\in [R,5R]$. For a point $x\in B=B(x_0,R)$, using \e{x9} and \e{x40}, we get
	\begin{align}
		|T(f\cdot & \ZI_{\ZR^d\setminus B^*})(x)-T(f\cdot \ZI_{\ZR^d\setminus B^*})(x_0)|\\
		&\lesssim \int_{|y-x_0|>R^*}|K(x,y)-K(x_0,y)||f(y)|dy\\
		&\le \sum_{k=0}^\infty\int_{2^{k+1}R^*\ge |y-x_0|>2^kR^*}\omega\left(\frac{|x-x_0|}{|x_0-y|}\right)\cdot \frac{1}{|x_0-y|^d}|f(y)|dy\\
		&\lesssim  \sum_{k=1}^\infty k\omega(2^{-k})\frac{1}{k |B(x_0,2^{k+1}r)|}\int_{B(x_0,2^{k+1}r)}|f(y)|dy\\
		&\lesssim[\omega] \sup_{A\supset B}\bigg(\log^{-1}\left(1+\frac{\mu(A)}{\mu(B)}\right)\langle f\rangle_{A}\bigg),
	\end{align}
	where $\sup$ is taken over all the balls $A\supset B$. Thus we obtain \e{y83}. 
\end{proof}
{Consider the classical singular operators on $\ZR^d$, corresponding to kernels
\begin{equation}\label{SZ}
	K(x)=\frac{\Omega(x)}{|x|^d},
\end{equation}
where $\Omega(x)$ satisfies
\begin{align*}
	&\Omega(\varepsilon x)=\Omega(x),\, \varepsilon>0,\quad \int_{S^{d-1}}\Omega(x)d\sigma=0,\\
	&\omega(\delta)=\sup_{\stackrel{|x-x'|\le \delta}{|x|=|x'|=1}}|\Omega(x)-\Omega(x')|,\quad \int_0^1\frac{\omega(t)}{t}dt<\infty.
\end{align*}
(see \cite{Ste1}, chap. 2.4, for precise definition). Under these conditions the operator \e{CZ} with the kernel $K(x,y)=K(x-y)$ generates a Calderón–Zygmund operator.
\begin{lemma}\label{L2}
	The singular operator, corresponding to the kernel \e{SZ} is uniformly vanishing.
\end{lemma}
\begin{proof}
	For an arbitrary ball $B=B(x_0,r)$ consider a ball $B'=B(x_0,R)$ with a radius $R>r$. Let $B''\supset B'$ be an arbitrary ball. For $x,x'\in B$ we have 
	\begin{align*}
		T(\ZI_{B''})(x)-T(\ZI_{B''})(x')&=\int_{B''}K(x-y)dy-\int_{B''}K(x'-y)dy\\
		&=\int_{x-B''}K(y)dy-\int_{x'-B''}K(y)dy\\
		&=\int_{(x-B'')\setminus (x'-B'')}K(y)dy-\int_{(x'-B'')\setminus (x-B'')}K(y)dy.
	\end{align*}
	Thus,
	\begin{equation}
		|T(\ZI_{B''})(x)-T(\ZI_{B''})(x')|\le \int_{(x-B'')\triangle (x'-B'')}|K(y)|dy.
	\end{equation}
	Observe that the set $(x-B'')\triangle (x'-B'')$ is in the $r$ neighborhood of the boundary $\partial(B'')$. That is
	\begin{equation}
		(x-B'')\triangle (x'-B'')\subset \{y:\, \dist(y,\partial B'')<r\}\subset (B(x_0, R-r))^c.
	\end{equation}
	Thus by passing to spherical coordinates, we can write
	\begin{equation*}
		\int_{(x-B'')\triangle (x'-B'')}|K(y)|dy=c_d\int_{S^{d-1}}\int_{t_1(s)}^{t_2(s)}t^{d-1}|K(t\cdot s)|dtds,
	\end{equation*}
	where $t_2(s)-t_1(s)<2r$ and $t_1(s)>R-r$. Then we obtain
	\begin{align}
		|T(\ZI_{B''})(x)-T(\ZI_{B''})(x')|&\le c_d \|\Omega\|_\infty\int_{S^{d-1}}\int_{t_1(s)}^{t_2(s)}\frac{t^{d-1}}{t^d}dtds\\
		&\le c_d\|\Omega\|_\infty \int_{S^{d-1}}\frac{2r}{t_1(s)}ds\\
		&\le c_d\|\Omega\|_\infty \frac{2r}{R-r}.
	\end{align}
	The latter can be less than any $\varepsilon>0$ by a suitable choice of $R$. This implies that $\OSC_B(T(\ZI_{B''}))\le \varepsilon$
	for any ball $B''\supset B'$. That means $T$ is uniformly vanishing.
\end{proof}
Observe that under the logarithmic Dini condition \e{x40}, the singular operator corresponding to kernel \e{SZ}  satisfies \e{y83} according to  \lem{L1}. Applying also \lem{L2}, we obtain
\begin{theorem}
	If a modulus of continuity $\omega(t)$ satisfies the condition \e{x40},
	then the singular operator $T$, corresponding to the kernel \e{SZ} is bounded on $\BMO$. Namely, for any $f\in L^r(\ZR^d)\cap \BMO(\ZR^d)$ we have $\|T(f)\|_\BMO\lesssim \|f\|_\BMO$. 
\end{theorem}
}
\subsection{Carleson type operators}
Now let $K(x,y)$ be a Calder\'on-Zygmund kernel with a modulus of continuity, satisfying \e{x40}. Consider the family of operators 
\begin{equation}
	T_\alpha f(x)=\int_{\ZR^d}K(x,y)g_\alpha(x)f(y)dy,
\end{equation}
where $\sup_\alpha \|g_\alpha\|_\infty\le 1$. The argument of \lem{L1} shows that the operators $T_\alpha$ satisfy \e{y83} uniformly. Thus, we can say that  $T(f)=\sup_{\alpha}|T_\alpha(f)|$
satisfies the hypothesis of \trm{T1}, provided $\|T\|_{L^r\to L^{r,\infty}}<\infty$. Hence, applying \trm{T11}, we conclude that $T$ maps boundedly $L^\infty(\ZR^d)$ into $\BMO(\ZR^d)$. This can be applied for Carleson type operators. Namely, 
recently P.~Zorin-Kranich \cite{Zor} considered operators
\begin{equation}
	Tf(x)=\sup_{Q\in \ZQ_d}\left|\int_{\ZR^d}K(x,y)e^{2\pi i Q(y)}f(y)dy\right|
\end{equation}
where $K(x,y):\ZR^d\times \ZR^d\to \ZR$ is a H\"older continuous Calder\'on-Zygmund kernel and $\ZQ^n$ denotes the family of polynomials in $d$ variables with degree at most $n$. It was proved that operators \e{x8} are bounded on $L^r(\ZR^d)$ for every $1<r<\infty$. In the case of Hilbert kernel $K(x,y)=(x-y)^{-1}$, this was earlier proved by V.~Lie in \cite{Lie}. The H\"older continuous Calder\'on-Zygmund kernels satisfy \e{x9} with $w(t)=t^\delta$, $\delta>0$, and so we will have \e{x40}. Thus, applying Theorems \ref{T1} and \ref{T11}, we arrive to the following
\begin{theorem}\label{T2}
	Let $\{Q_\alpha\}$ be a family of polynomials from $ \ZQ^n$ and let $K(x,y)$ be a Calder\'on-Zygmund kernel with the modulus of continuity $w(t)=t^\delta$, $\delta>0$. Then for the operator 
	\begin{equation}
		Tf(x)=\sup_{\alpha}\left|\int_{\ZR^d}K(x,y)e^{2\pi i Q_\alpha(y)}f(y)dy\right|
	\end{equation}
	we have bounds \e{y44} and \e{y45}. In particular, the operator $T$ maps $L^\infty$ into $\BMO$ and
	\begin{equation*}
		\|T(f)\|_{\BMO(\ZR^d)}\lesssim \|f\|_\infty.
	\end{equation*}
\end{theorem}
\begin{remark}
	The same argument is applicable for Calder\'on-Zygmund operators on bounded intervals, including periodic cases. In particular, the result of \trm{T2} also holds 
	for the maximal Fourier partial sums in trigonometric and Walsh systems.
\end{remark}
\subsection{Martingales}
Let our measure space $X$ be finite and let $\zB_n$ be a filtration defined in $X$.  We suppose that $\zB=\cup_{n\ge 0}\zB_n$ is the corresponding ball-basis and it satisfies the doubling condition. Given $f\in L^1(X)$ consider the martingale difference 
\begin{equation}
	\Delta_Af(x)=\sum_{B:\, \pr(B)=A}\left(\frac{1}{\mu(B)}\int_Bf-\frac{1}{\mu(A)}\int_Af\right)\ZI_B(x).
\end{equation}
corresponding to our filtration. Recall the definition of Burkholder's \cite{Burk} martingale transform, namely the operators of the form
\begin{equation}\label{x43}
		M_\varepsilon f(x)=\sum_{A\in \zB}\varepsilon_{A}\Delta_A f(x),\quad |\varepsilon_A|\le 1.
\end{equation}
 We consider also other operators related to \e{x43}:
\begin{align}
	&M_{\varepsilon,n}f(x)=\sum_{A\in \cup_{0\le k\le n}\zB_k}\varepsilon_{A}\Delta_A f(x),\label{x45}\\
	&M_\varepsilon^* f(x)=\sup_{n}|M_{\varepsilon,n}f(x)|,\label{x46}\\
	&M_\varepsilon^+ f(x)=\sup_{n}M_{\varepsilon,n}f(x),\label{x47}\\
	&M_\varepsilon^- f(x)=\inf_{n}M_{\varepsilon,n}f(x),\label{x48}\\
	&Sf(x)=\left(\sum_{A\in \zB}\left|\Delta_A f(x)\right|^2\right)^{1/2}.\label{x49}
\end{align}
Letting $T$ to be one of the these operators, one can check that
\begin{equation}\label{x53}
	\OSC_B\big(T (f\cdot \ZI_{X\setminus B^*})\big)=0.
\end{equation}
for any ball $B\in \cup_{k\le n}\zB_k$. {Indeed, if $A\subseteq B^*$, then $\Delta_A(f\cdot \ZI_{X\setminus B^*})$ is identically zero. Otherwise we have $B^*\subsetneq A$ and therefore $\Delta_A(f\cdot \ZI_{X\setminus B^*})$ is constant on $B^*$. All these imply that for all above operators $T (f\cdot \ZI_{X\setminus B^*})$ is constant on $B$. Thus we get \e{x53}.} This means all these operators are $\BO$ operators, since \e{y83} holds for any modulus of continuity $\omega$. Also, it is easy to check that all of them are  uniformly vanishing. It is well known that these operators are bounded on $L^p$ spaces if $1<p\le \infty$ even if the filtration is not doubling (see \cite{Chao1,Chao2}). Thus we can say that all those operators satisfy bound \e{y30} and are bounded on $\BMO(X)$.
\subsection{Wavelet type systems}
Recall the definition of biorthogonal wavelet type system $\Phi=\{\phi_k(x), \psi_n(x)\}_{n=0}^\infty$ on $[0,1]$. We will use the function
\begin{equation}\label{u51}
	\xi(x)=\frac{1}{(1+|x|)^{1+\delta}},\quad 0<\delta<1,
\end{equation}
as well as the notations 
\begin{align}
	&t_1=\frac{1}{2},\quad t_k=\frac{2j-1}{2^{n+1}},\quad k\ge 2,\\
	&\text{where }	k=2^n+j,\, 1\le j\le 2^n,\,n=0,1,2,\ldots.\label{u52}
\end{align}
Hence we suppose 
\begin{align}
	&\phi_0(x)=\psi_0(x)\equiv 1,\label{u65}\\
	&|\phi_k(x)|+|\psi_k(x)|\le c2^{n/2}\cdot \xi\big(2^n(x-t_k)\big),\quad k\ge 1,\label{u43}\\
	&|\phi_k(t)-\phi_k(t')|+|\psi_k(t)-\psi_k(t')|\\
	&\qquad\qquad\qquad\le c2^{n/2}(2^n|t-t'|)^\alpha \cdot \xi\big(2^n(t-t_k)\big)\text { if }|t-t'|\le 2^{-n}, \label{u44}
\end{align}
where $0<\alpha\le 1$.  It is well known that the expansion
\begin{equation*}
	\sum_{n=0}^\infty \langle f,\psi_n\rangle \phi_n(x),\quad \langle f,\psi_n\rangle =\int_0^1f(t)\psi_n(t)dt,
\end{equation*}
of any function $f\in L^p$, $1<p<\infty$, in a wavelet type system converges in $L^p$ unconditionally. Moreover, the wavelet type systems share the following property of unconditional bases of $L^p$:  for any sequence $\Lambda=\{\lambda_k\}$, $\|\Lambda\|=\sup_k|\lambda_k|\le 1$, and a function $f\in L^p$, $1<p<\infty$, the series $T_\Lambda(f)=\sum_k\lambda_k\langle f,\psi_k\rangle \phi_k$ converges in $L^p$ and $\left\|T_\Lambda(f)\right\|_p\lesssim \|f\|_p$.
This property can be established by a well-known argument (see for example \cite{HeWe}), proving that $K_\lambda(x,t)=\sum_k \lambda_k\phi_k(x)\phi_k(t)$ is a Calder\'on-Zygmund kernel, namely 
\begin{align}
	&\left|K_\Lambda(x,y)\right|+\left|K_\Lambda(y,x)\right|\lesssim \frac{1}{|x-y|},\label{y1}\\
	&|K_\Lambda(x,y)-K_\Lambda(x,y')|+|K_\Lambda(y,x)-K_\Lambda(y',x)|\lesssim \frac{|y-y'|^\beta }{|x-y|^{1+\beta}},\quad \beta>0,\label{y2}
\end{align}
whenever $|x-y|>2|y-y'|$. Moreover, if in addition $\{\phi_k\}$ is a basis in $L^2[0,1]$, then it becomes an unconditional basis in all spaces $L^p[0,1]$, $1<p<\infty$ (see for example \cite{HeWe}). Concerning to the unconditional wavelet bases in $L^p$, we can also refer the readers to the papers \cite{Mey, Woj, Grip, Wol, StSj}. Examples of wavelet type systems are many classical wavelet systems, orthonormal spline systems on $[0,1]$. Wavelet type systems were considered also on atomic Hardy space $H^1[0,1]$. The existence of an unconditional basis in $H^1[0,1]$ first was proved by Maurey \cite{Mau}. Later L. Carleson  \cite{Car} constructed an explicit sequence in $\BMO$ whose biorthogonal system forms an unconditional basis in $H^1$.  Wojtaszczyk  \cite{Woj1} proved that Franklin system (the picewise linear orthonormal spline system on $[0,1]$) is an unconditional basis for $H^1[0,1]$. This was the first explicit example of orthonormal system, which is an unconditional basis for $H^1[0,1]$. 

Hence we can say that the operators $T_\Lambda$ with $\|\Lambda\|\le 1$ form a family of Calder\'on-Zygmund operators with uniformly bounded constants. Clearly, those also are uniformly vanishing. Indeed, from \e{u65} it follows that $T_\lambda(\ZI_{[0,1]})\equiv \lambda_0$ and then, $\OSC_{[0,1]}(T_\lambda(\ZI_{[0,1]}))=0$. Choose an arbitrary family of sequences $\Lambda_\alpha=\{\lambda_k^\alpha\}$ with $\|\Lambda_\alpha\|\le 1$. This generate the maximal operators
\begin{equation}\label{r48}
	T^*(f)=\sup_\alpha |T_{\Lambda_\alpha}(f)|,\quad T^\pm(f)=\sup_\alpha (\pm  T_{\Lambda_\alpha}(f))
\end{equation}
which satisfy the hypothesis of \trm{T5}. Thus we conclude 
\begin{theorem}
	Operators in \e{r48} satisfy the bound \e{y30}. In particular, they map $\BMO[0,1]$ into itself.
\end{theorem}
Recall that $\BMO[0,1]$ is the dual space of $H^1[0,1]$. So by the duality argument, for $f\in H^1$ and $g\in \BMO$ we have
\begin{equation}
	\langle T_\Lambda(f),g\rangle= \langle f,\bar T_\Lambda(g)\rangle\le \| f\|_{H^1}\|\cdot \|\bar T_\Lambda(g)\|_\BMO\lesssim  \| f\|_{H^1}\cdot  \| g\|_{\BMO},
\end{equation}
where $\bar T_\Lambda$ denotes the adjoint operator of $T_\Lambda$. This implies the estimate $\|T_\Lambda(f)\|_{H^1}\lesssim \| f\|_{H^1}$, which is the characteristic bound for the unconditional basis in $H^1$. Hence we can get the following.
\begin{corollary}
	If a biorthogonal wavelet type system is a basis in $L^2[0,1]$, then it is an unconditional basis in $H^1[0,1]$.
\end{corollary}

\end{document}